\documentclass[a4paper,12pt,reqno]{amsart}

\usepackage{graphicx}
\usepackage{psfrag}
\usepackage{wrapfig}
\usepackage[pdftex]{thumbpdf}
\usepackage[latin1]{inputenc}
\usepackage[T1]{fontenc}
\usepackage{float}
\usepackage{enumerate}
\usepackage[reqno]{amsmath}     
\usepackage{amssymb}            
\usepackage{mathrsfs}           
\usepackage{amsxtra}
\usepackage{hyperref}
\usepackage{color}
\usepackage{diagrams}

\newtheorem{thm}{Theorem}[section]
\newtheorem{lem}[thm]{Lemma}

\newtheorem{prop}[thm]{Proposition}
\newtheorem{cor}[thm]{Corollary}

\theoremstyle{remark}

\theoremstyle{definition}

\numberwithin{equation}{section}

\def\F{{\mathcal{F}}}   
\def\J{{\mathcal{J}}}   
\def\Av{{\mathcal{A}}}  
\def\Cv{{\mathcal{C}}}  
\def\Sv{{\mathcal{S}}}  
\def\Pv{{\mathcal{P}}}  
\def\Ev{{\mathcal{E}}}  
\def\Ov{{\mathcal{O}}}  
\newcommand{\Crit}{\operatorname{Crit}}  
\newcommand{\A}{\operatorname{A}}   
\newcommand{\I}{\operatorname{I}}   

\def\C{{\mathbb{C}}}
\def\D{{\mathbb{D}}}
\def\N{{\mathbb{N}}}

\def\Circ{{\mathbb{S}}}

\newcommand{\eps}{\varepsilon}
\newcommand{\dist}{\operatorname{dist}}  
\newcommand{\e}{\operatorname{e}}

\newcommand{\M}{\operatorname{M}}  
\newcommand{\m}{\operatorname{m}}  

\newcommand{\cl}{\overline}
\newcommand{\abs}[1]{\left|#1\right|}
\newcommand{\set}[1]{\left\{#1\right\}}

\begin{document}
\title[Poincar\'{e} functions with spiders' webs]
{Poincar\'{e} functions with spiders' webs}
\author{Helena Mihaljevi\'{c}-Brandt}
\address{Mathematisches Seminar, Christian-Albrechts-Universit\"{a}t zu Kiel,
24118 Kiel, Germany}
\email{helenam@math.uni-kiel.de}
\author{J\"{o}rn Peter}
\thanks{The second author has been
supported by the Deutsche Forschungsgemeinschaft, Be 1508/7-1.}
\address{Mathematisches Seminar, Christian-Albrechts-Universit\"{a}t zu Kiel,
 24118 Kiel, Germany}
\email{peter@math.uni-kiel.de}

\maketitle

\begin{abstract}
For a polynomial $p$ with
a repelling fixed point $z_0$, we consider
\emph{Poincar\'{e} functions} of $p$ at $z_0$, i.e.
entire functions $L$ which satisfy
$L(0)=z_0$ and $p(L(z))=L( p'(z_0) \cdot z)$ for all $z\in\C$.
We show that if the component of the Julia set of $p$ that contains $z_0$ equals $\{z_0\}$,
then the (fast) escaping set of $L$ is a \emph{spider's web}; in particular it is connected.
More precisely, we classify all linearizers of polynomials with regards to the spider's web
structure of the set of all points which
escape faster than the iterates
of the maximum modulus function at a sufficiently large point $R$.
\end{abstract}

\section{Introduction}

Let $f$ be a transcendental entire function.
With the fundamental work of Eremenko \cite{eremenko_1},
the \emph{escaping set}
\begin{align*}
 \I(f):=\{z\in\C:f^n(z)\to\infty\text{ as }n\to\infty\}
\end{align*}
has become an intensively studied object in transcendental holomorphic dynamics.
Since then, much progress has been achieved in exploring the topological and
dynamical properties of the escaping set and some of its subsets
(for some results, see \cite{mihaljevic-brandt,rempe_1,rs_4,rs_1,rs_2,rrrs}).

Rippon and Stallard discovered that the
\emph{fast escaping set} $\A(f)$, which was
originally introduced by Bergweiler and Hinkkanen \cite{bergweiler_hinkkanen},
shares many significant features with $\I(f)$.
If we set $M(f,r):=\max_{\abs{z}=r}\abs{f(z)}$ and choose any constant $R$ such that
\begin{align}
\label{eqn_R}
\M(f,r)>r\ \text{whenever} \ r\geq R,
\end{align}
the fast escaping set of $f$ can be described as
\begin{align*}
 \A(f)=\bigcup_{l\in\N} \A_R^{-l}(f),
\end{align*}
where $\A_R^{l}(f)$ are the so-called \emph{level sets}, defined by
\begin{align*}
 \A_R^{l}(f):=\lbrace z\in\C: \vert f^{n+l}(z)\vert\geq \M^n(R), n\geq \max\{0,-l\}\rbrace.
\end{align*}
(Throughout the article $\M^n$
denotes the n-th iterate of the maximum modulus function.)

Recently, Rippon and Stallard \cite{rs_2,rs_4} introduced the concept of an
\emph{(infinite) spider's web}.
This is a connected set $E\subset\C$ with the
property that there exists a sequence of increasing simply-connected domains
$(G_n)$ whose union is all of $\C$ such that 
$\partial G_n\subset E$ for all $n$.
Functions whose (fast) escaping set is a spider's web
have some strong dynamical properties.
For instance, every such function has only bounded Fatou components
and there exists no curve to $\infty$ on which $f$ is bounded
(compare \cite{rs_2}).
In particular, the set of singular values of $f$
must be unbounded. (For precise definitions see
Section \ref{sec_prel}).

In \cite{rs_2}, various sufficient criteria are presented
such that $\I(f)$ and $\A(f)$ is a spider's web.
Primarily, this is the case whenever the set
\begin{align*}
 \A_R(f):=\A_R^0(f)
\end{align*}
 is a spider's web for any $R$ as in (\ref{eqn_R}).

In this paper, we present a large and interesting
class of functions whose escaping set is a spider's web,
namely, Poincar\'{e} functions of certain polynomials.
To make this precise, let $p$ be a polynomial with
a repelling fixed point $z_0$ (i.e. $p(z_0)=z_0$ and $\abs{p'(z_0)}>1$).
Then there exists an entire function $L$ called a
\emph{Poincar\'{e} function} or a \emph{linearizer of $p$ at $z_0$}
which satisfies
\begin{align*}
L(0)=z_0 \quad\text{and}\quad p(L(z))=L( p'(z_0) \cdot z) \;\text{for all}\; z\in\C.
\end{align*}
In the above functional equation, we can iterate the function $p$;
this already indicates that the analysis of a linearizer strongly depends on the
dynamical properties of $p$.
However, $L$ does not only depend on $p$ but also
on $z_0$ and $p'(z_0)$ which makes linearizers good candidates for
constructing functions with various interesting analytical properties
(see e.g. Section \ref{subsection_linearizers_sing_values}).
Furthermore, they are naturally good candidates for
constructing gauge functions to estimate the Hausdorff measure
of escaping and Julia sets of exponential functions (see \cite{peter}).

It was conjectured by Rempe that the escaping set of a linearizer of a
quadratic polynomial for which the critical point escapes is a spider's web.
In this article, we show that this is true; moreover, we classifiy
all linearizers of polynomials
corresponding to whether the sets $\A_R(L)$  are spiders' webs or not.

\begin{thm}
\label{thm_main}
 Let $p$ be a polynomial of degree $d\geq 2$,
let $z_0$ be a repelling fixed point of $p$ and let $L$ be a linearizer of $p$ at $z_0$.
If $R$ satisfies (\ref{eqn_R}) then
$\A_R(L)$ is a spider's web if and only if the component
of $\J(p)$ which contains $z_0$ equals $\{z_0\}$.
\end{thm}

Since polynomials for which all critical points converge to $\infty$
have totally disconnected Julia sets \cite[p.85]{fatou},
we obtain, using \cite[Theorem 1.4]{rs_2}, the following corollary which also
implies Rempe's conjecture.

\begin{cor}
\label{cor_1}
 Let $p$ be a polynomial of degree $d\geq 2$ for which all critical
points escape and let $L$ be a linearizer of $p$.
Assume that $R$ satisfies (\ref{eqn_R}).
Then each of the sets $\A_R(L)$, $\A(L)$ and $\I(L)$ is a spider's web.
In particular, this is true whenever
$p(z)=z^2 +c$ and $c$ lies outside the Mandelbrot set.
\end{cor}
We believe that the dichotomy established in Theorem \ref{thm_main} for the sets
$\A_R(L)$ also extends to the sets $\A(L)$ and $\I(L)$. However, we were not able to
prove this. For the fast escaping set, such a result would follow
if every continuum in $\A(f)$
(or every 'loop') would be contained in some level set $\A_R^l(f)$,
which we also believe to be true (compare Question $2$ and $3$ in \cite{rs_2}).

In the proof of Theorem \ref{thm_main},
we establish spiders' webs by proving that the corresponding linearizers
grow regularly and that there
exist simple closed curves arbitrary close to $0$ on which the minimum
modulus grows fast enough.

Since the order of a linearizer of a quadratic polynomial is given by
$\log 2/\log\abs{p^{'}(z_0)}$, we obtain for any given $\rho\in (0,\infty)$
a linearizer of order $\rho$ whose escaping set is a spider's web.

\subsection*{Acknowledgements}
We would like to thank Adam Epstein
for drawing our attention to Poincar\'{e} functions
and for pointing out various interesting phenomena
related to them.
Moreover, we want to thank Walter Bergweiler, Jean-Marie Bois, Lasse Rempe,
Phil Rippon and Gwyneth Stallard for many interesting discussions.

\section{Preliminaries}
\label{sec_prel}
The complex plane, the Riemann sphere and the unit disk are denoted by
$\C$, $\widehat{\C}:=\C\cup\lbrace\infty\rbrace$ and
$\D$, respectively. The circle at $0$ with radius $r$ will be denoted
by $\Circ_r$. We write $\D_r(z)$ for the Euclidean disk of radius $r$
centred at $z$.

If not stated differently, we will assume throughout the article
that $f:\C\to\C$ is a non-constant, non-linear entire function; so $f$ is either a
polynomial of degree $\geq 2$ or
 a transcendental entire map.

Let $C\subset\C$ be a compact set. The \emph{maximum modulus} $\M(f,C)$ and
the \emph{minimum modulus} $ \m(f,C)$
of $f$ relative to $C$ are defined to be
\begin{align*}
 \M(f,C):=\max_{z\in C} \vert f(z)\vert\quad
\text{and}\quad \m(f,C):=\min_{z\in C}\vert f(z)\vert.
\end{align*}
In the case when $C=\Circ_r$ we will simplify the notation
by writing $\M(f,r)$ and $\m(f,r)$ for
$\M(f,\Circ_r)$ and $\m(f,\Circ_r)$, respectively.
Finally, recall that the \emph{order} of $f$ is defined as
\begin{align*}
 \rho(f):=\limsup_{r\to\infty}\frac{\log\log \M(f,r)}{\log r}.
\end{align*}

\subsection{Background on dynamics of entire maps}
We denote by $\Crit(f):=\lbrace z\in\C: f'(z)=0\rbrace$ the
set of \emph{critical points}, by $\Cv(f):=f(\Crit(f))$ the
set of \emph{critical values}, and by $\Av(f)$ the set
of all \emph{(finite) asymptotic values} of $f$.
The elements of $\Sv(f)=\cl{\Cv(f)\cup \Av(f)}$ are called
\emph{singular values} of $f$, and $\Sv(f)$ can be characterized as
the smallest closed subset of $\C$
such that $f:\C\setminus f^{-1}(\Sv(f))\to\C\setminus \Sv(f)$
is a covering map.
If $f$ is a polynomial then
$\Av(f)=\emptyset$ and $\Cv(f)$ is finite, so in this case, $\Sv(f)=\Cv(f)$.
The \emph{postsingular set} of $f$ is
defined to be $\Pv(f):=\cl{\bigcup_{n\geq 0} f^{n}(\Sv(f))}$.

Denote by $f^n$ the $n$-th iterate of $f$.
A point $w\in\C$ is said to be \emph{exceptional under $f$} if its backward orbit,
i.e., the set of all points $z$ which are mapped to $w$ by some $f^n$,
is finite. The set of all exceptional values of $f$ will be denoted by $\Ev(f)$.
It is well known that $\Ev(f)$ contains at most one point.
We write $\Ov(f)$ for the set of all (finite) \emph{omitted values} of $f$.

If $z$ is a periodic point of $f$ of period $n$,
we call $\mu(z):=(f^n)^{'}(z)$ the \emph{multiplier} of $z$.
A periodic point $z$ is called
\emph{attracting} if $0\leq \vert \mu(z)\vert<1$,
\emph{indifferent} if $\vert \mu(z)\vert=1$ and
\emph{repelling} if $\vert \mu(z)\vert>1$.
An attracting periodic point $z$ is called
\emph{superattracting} if $\mu(z)=0$.


The \emph{Fatou set} $\F(f)$ of $f$ is the set of all
points that have a neighbourhood in which the iterates of $f$
form a normal family; the \emph{Julia set} $\J(f)$ is defined to be
$\C\setminus\J(f)$.
For a point $z\in\J(f)$ we denote by $\J_z(f)$ the component
of $\J(f)$ that contains $z$.

\subsection{Poincar\'{e} functions}

Let $z_0$ be a repelling fixed point of $f$ with multiplier $\lambda$.
By the K\oe{}nigs Linearization Theorem \cite[Theorem 8.2]{milnor},
there exists a holomorphic function $l$ defined in a neighbourhood of
$0$ such that $l(0)=z_0$ and locally, $l^{-1}\circ f\circ l(z)=\lambda z$.
It was observed already by Poincar\'{e} that $l$ continues to a holomorphic
function $L$ on the entire complex plane, meaning there exists an entire
map $L$ such that $L(0)=z_0$ and
\begin{align}
\label{func_eq}
 f(L(z))=L(\lambda z)
\end{align}
for all $z\in\C$.
Every such map is called \emph{linearizer} or \emph{Poincar\'{e} function} of $f$
at $z_0$.
A linearizer is unique up to a constant. More precisely,
if $L$ satisfies (\ref{func_eq}), then so does $L_c:z\mapsto L(cz)$
for every $c\in\C^{*}$, and every solution of the equation (\ref{func_eq}) is of this form.
If $L_1$ and $L_2$ are two linearizers then we say that they have the
same \emph{normalization} if $L_1^{'}(0)= L_2^{'}(0)$.
We say that $L$ has the \emph{standard normalization} if $L^{'}(0)=1$.

\begin{prop}
\label{prop_lin_conf_conj}
 Let $f_1$ and $f_2$ be entire functions, and assume that there
exists a conformal map $\varphi(z)=az+b$ such that
\begin{align*}
 f_2=\varphi^{-1}\circ f_1\circ\varphi
\end{align*}
everywhere in $\C$. If $L_1$ and $L_2$ are linearizers of $f_1$ and $f_2$ at
$z_1$ and $z_2=\varphi^{-1}(z_1)$, respectively, with
the same normalization, then
\begin{align*}
 L_2=\varphi^{-1}\circ L_1\circ (\varphi-b),
\end{align*}
where $(\varphi-b)(z):=\varphi(z)-b =az$.
\end{prop}

\begin{proof}
Consider the function $\widetilde{L}(z):=\varphi^{-1}\circ L_1\circ (\varphi-b)(z)$.
Then $\widetilde{L}$ satisfies
\begin{eqnarray*}
 f_2\circ \widetilde{L}(z/\lambda)&=& f_2\circ \varphi^{-1}\circ L_1
\circ (\varphi-b)(z/\lambda)
=\varphi^{-1}\circ f_1\circ L_1(az/\lambda)\\
&=& \varphi^{-1} L_1(az) = \widetilde{L}(z).
\end{eqnarray*}
 Since $f_1$ and $f_2$ are conformally conjugate, the multipliers at
$z_1$ and $z_2$ coincide, hence
$\widetilde{L}$ is a linearizer of $f_2$ at $\varphi^{-1}(z_1)=z_2$.
Furthermore, $\widetilde{L}^{'}(0)=L_1^{'}(0)$, hence
$L_1$ and $\widetilde{L}$ have the same normalization, yielding
$L_2= \widetilde{L}$.
\end{proof}

In many dynamical settings,
conformal conjugacies produce no relevant dynamical consequences,
hence it is natural to ask the following: Assume that
$f_1$ and $f_2$ are as in Proposition \ref{prop_lin_conf_conj}
(so $f_1$ and $f_2$ are conformally conjugate entire functions)
and let $L_1$ be a
linearizer of $f_1$. Does there exist a linearizer $L_2$ of $f_2$ which is
conformally conjugate to $L_1$ (and hence has the same dynamics)?
In general, the answer is no. If namely such a linearizer $L_2$ would exist,
then a corresponding conjugacy, say $\psi$, would map $\Sv(L_1)$ bijectively onto
$\Sv(L_2)$, which
turns out to be equivalent to the condition
\begin{align}
\label{eqn_psi}
 \psi(\Pv(f_1))=\Pv(f_2)
\end{align}
(see Proposition \ref{prop_sing_values_lin}).
Since $\varphi$ conjugates $f_1$ and $f_2$, it already satisfies (\ref{eqn_psi}),
so in particular, the map $\psi^{-1}\circ\varphi$ is a conformal
automorphism of $\C$ that fixes the set $\Pv(f_1)$.

Now if $Z$ is an arbitrary finite subset of $\C$ with at least two elements,
then $G_Z:=\lbrace h(z)=az+b: a\in\C^{*},b\in\C,h(Z)=Z\rbrace$
is a finite group and one can easily check that the map
$G_Z\to\C^{*}, az+b\mapsto a$ is an injective group-homomorphism.
Hence $G_Z$ is isomorphic to a finite subgroup of $\C^{*}$,
which must be a cyclic group generated by a root of unity.
So every such $G_Z$ is generated by a map of the form $z\mapsto \exp(2\pi ik/n) z +b$
with coprime $k$ and $n$ and $n\leq\vert Z\vert$.
This allows to phrase necessary geometric conditions on a finite set $Z$ such that
$G_Z$ is not trivial. It is clear that such conditions are
rather strong; e.g., if $z\mapsto \exp(2\pi ik/n) z +b$
is a generator of $G_Z$ and
$p$ its (unique) fixed point in $\C$
then all elements of $Z$ must lie on
$r$ circles centred at $p$, where $r\cdot n\leq\vert Z\setminus \lbrace p\rbrace\vert$.
To give an explicit dynamical example,
one can consider the unique real
parameter $c$, for which $f(z):=z^2 +c$ has a superattracting
cycle of period three; one easily sees that
$G_{\Pv(f)}$ is trivial.

However, triviality of $G_{\Pv(f_1)}$ implies $\psi\equiv\varphi$.
So if $\varphi(z)=az+b$, then by Proposition \ref{prop_lin_conf_conj},
every linearizer of $f_2$ is of the form
\begin{align*}
 L_2(z)=\varphi^{-1}\circ L_1\circ c(\varphi - b)
\end{align*}
for some $c\in\C^{*}$, and no such map can be conformally conjugate to
$L_1$ via $\varphi$ whenever $b\neq 0$ (and $c\neq 1$).

Before the end of this paragraph let us observe that
one can iterate $f$ inside the functional equation and obtain
\begin{align}
\label{eqn_iter_fct_eq}
 f^n\circ L(z)=L\circ \lambda^n (z)
\end{align}
as an iterated version of (\ref{func_eq}),
where $\lambda^n$ denotes the function $z\mapsto\lambda^n z$.

The growth of the function $f$ and a linearizer $L$ are related in the
following sense: If $f$ is transcendental entire then
$L$ has infinite order. If $f$ is a polynomial then
$\rho(L)=\log d/\log\vert\lambda\vert$.

\subsection{Polynomial dynamics near $\infty$ and repelling fixed points}

If $p$ is a polynomial,
the Julia set of $p$ is compact
and $\I(p)$ is an open connected subset of $\F(p)$;
moreover, it is simply-connected if and only if $\J(p)$ is connected.
This property is equivalent to the relation
$\Cv(p)\cap\I(p)=\emptyset$
\cite[Lemma 9.4, Theorem 9.5]{milnor}.

Near $\infty$, the iterates of a polynomial behave in the following
simple way.
\begin{prop}
\label{prop_pol_1}
 Let $p(z)=\sum_{n=0}^d a_n z^n$ be a polynomial of degree $d\geq 2$.
Then for any $\eps>0$ there exists $R_{\eps}>0$ such that
for every $z$ with $\vert z\vert>R_{\eps}$, we have
\begin{align*}
 (1-\eps)\vert a_d\vert\cdot\vert z\vert^d\leq\vert p(z)\vert
\leq(1+\eps)\vert a_d\vert\cdot \vert z\vert^d,
\end{align*}
and $R_{\eps}\to\infty$ as $\eps\to 0$.

If $\eps$ is chosen small enough such that
$(1-\eps)\vert a_d\vert R_{\eps}^{d-1}>1$, then
\begin{align*}
((1-\eps)\vert a_d\vert)^{q_n(d)}\cdot\vert z\vert^{d^n}\leq\vert p^n(z)\vert\leq
((1+\eps)\vert a_d\vert)^{q_n(d)}\cdot \vert z\vert^{d^n}
\end{align*}
 for all $n\in\N$ and all $z\in\C$ with $\vert z\vert >R_{\eps}$,
where $q_n(z):=(z^n-1)/(z-1)=z^{n-1} + \dots +z+1$.
\end{prop}

\begin{proof}
The first statement is elementary and well-known.

 Note that we have chosen $\eps$ sufficiently small
such that $\vert z\vert>R_{\eps}$ implies
$\vert p(z)\vert>R_{\eps}$. We will prove the statement inductivly. So for
$n=1$ we have $q_1(z)=1$ and the claim follows from the first part.
For the iterate $p^{n+1}(z)=p(p^n(z))$ we then obtain
\begin{eqnarray*}
\vert p(p^n(z))\vert&\leq& (1+\eps)\vert a_d\vert \vert p^n(z)\vert^d
\leq (1+\eps)\vert a_d\vert \left[((1+\eps)\vert a_d\vert)^{q_n(d)} \vert z\vert^{d^n}\right]^d\\
&=&((1+\eps)\vert a_d\vert)^{d\cdot q_n(d)+1} \vert z\vert^{d^{n+1}}
=((1+\eps)\vert a_d\vert)^{q_{n+1}(d)} \vert z\vert^{d^{n+1}}
\end{eqnarray*}
 as well as
\begin{eqnarray*}
\vert p(p^n(z))\vert&\geq& (1-\eps)\vert a_d\vert \vert p^n(z)\vert^d
\geq (1-\eps)\vert a_d\vert \left[((1-\eps)\vert a_d\vert)^{q_n(d)} \vert z\vert^{d^n}\right]^d\\
&=&((1-\eps)\vert a_d\vert)^{d\cdot q_n(d)+1} \vert z\vert^{d^{n+1}}
=((1-\eps)\vert a_d\vert)^{q_{n+1}(d)} \vert z\vert^{d^{n+1}}.
\end{eqnarray*}
\end{proof}

Near a repelling fixed point of $p$, we can make
the following statement on the escaping set $\I(p)$.

\begin{prop}
\label{prop_curve_escapingset}
 Let  $z_0$ a repelling fixed point of $p$.
For every $\delta>0$ there exists a simple closed
curve $\gamma_{\delta}\subset \D_{\delta}(z_0)\cap \I(p)$
 around $z_0$ if and only if $\J_{z_0}(p)=\{z_0\}$.
\end{prop}

\begin{proof}
Let us first assume that for every $\delta>0$
there exists a simple closed curve $\gamma_{\delta}\subset \I(p)$
around $z_0$ such that $\dist(z_0,\gamma_{\delta})<\delta$. Then $\J_{z_0}(p)$
is contained in the interior of every $\gamma_{\delta}$, hence it must consist
of a single point.


If $\J_{z_0}(p)=\{z_0\}$, then for every $\delta>0$ there exist
open, non-empty disjoint sets $U_{\delta}$
and $V_{\delta}$ such that $\J(p)\subset U_{\delta}\cup V_{\delta}$,
$\J_{z_0}(p)\subset U_{\delta}$
and $\dist(z_0,U_{\delta})<\delta/2$. Furthermore, we can assume $U_{\delta}$ to be connected;
otherwise, we replace $U_{\delta}$ by the connected component of $U_{\delta}$ that
contains $z_0$, which is also an open set.
By the Plane Separation Theorem \cite[Chapter VI, Theorem 3.1]{whyburn}, there
exists a simple closed curve $S_{\delta}$ which separates $z_0$ from $\J(p)\cap V_{\delta}$
such that $S_{\delta}\cap\J(p)=\emptyset$ and every point in $\J(p)\cap U_{\delta}$ is at
distance less than $\delta/2$ from $S_{\delta}$. Hence, $\dist(S_{\delta},z_0)<\delta$ and
$S_{\delta}\subset\F(p)$. Moreover, the component of $\F(p)$ which contains $S_{\delta}$
must be $\I(p)$ since every bounded component of the Fatou set is simply-connected.
\end{proof}

\section{The set of singular values of a linearizer}
\label{subsection_linearizers_sing_values}

If not stated differently, we will
assume throughout this section
that $f$ is an entire function, $z_0$ a repelling fixed point
of $f$ and $L$ a linearizer of $f$ at $z_0$.
We begin with a simple
connection between exceptional values of
$f$ and omitted values of $L$.

\begin{prop}
\label{omitvalues}
The sets $\Ov(L)$ and $\Ev(f)\setminus\set{z_0}$ are equal.
\end{prop}

\begin{proof}
Since $L(0)=z_0$, the point $z_0$ is never an omitted value of $L$.
If $a\in\C\setminus \Ev(f)$, then the backward
orbit of $a$ has infinitely many elements.
Since $L$ omits at most one finite value, the backward orbit
of $a$ under $f$ intersects $L(\C)$, i.e.,
there exists $n\in\N$ and $w\in\C$ with $L(w)\in f^{-n}(a)$.
This means $a=f^n(L(w))=L(\lambda^nw)$, so $a\notin \Ov(L)$.
This proves $\Ov(L)\subset \Ev(f)\setminus\set{z_0}$.

Now let $a\in\C\setminus \Ov(L)$. If $a=z_0$, then we are done.
So suppose that $a\neq z_0$. Then there exists
$z\neq 0$ with $L(z)=a$. By the iterated functional equation,
$L(z/\lambda^j)\in f^{-j}(a)$. Since $z\neq 0$ and $L$
is injective in a neighborhood of $0$, the backward orbit of $a$ under $f$
has infinitely many elements.
\end{proof}

Next, we will show that
the postsingular set of $f$ and
the set of singular values $L$ coincide.
This seems to be well-known
(and to us, the main parts of the proof have been presented by A. Epstein), but we
could not find a reference, which is why we include a proof.

\begin{prop}
\label{prop_sing_values_lin}
The following relations are true:
\begin{itemize}
 \item[$(i)$] $\Cv(L)=\bigcup_{n\geq 0} f^n(\Cv(f))\setminus \Ev(f)$.
 \item[$(ii)$] $\Sv(L)=\Pv(f)$.
\end{itemize}
\end{prop}

\begin{proof}
Let $w=L(z)\in \Cv(L)$, in particular $w\notin \Ov(L)$.
Since $L'(0)\neq 0$, we have $w\neq z_0$.
It follows from Proposition \ref{omitvalues} that $w\notin \Ev(f)$.
Differentiating the iterated functional equation yields
\begin{align*}
 0=(f^n)'(L(z/\lambda^n))\cdot L'(z/\lambda^n)\cdot\frac{1}{\lambda^n}.
\end{align*}
Denote by $\Crit(f)$ the set of critical points of $f$.
Since $L'(z/\lambda^n)\neq 0$ if $n$
is large enough, it follows that $L(z/\lambda^n)\in \Crit(f^n)$.
Since $\Crit(f^n)=\bigcup_{k=0}^{n-1}f^k(\Crit(f))$ by the chain rule,
there exists some $k\leq n-1$ with $L(z/\lambda^n)=f^k(y)$, where $y\in \Crit(f)$.
It follows that
\begin{align*}
 w=L(z)=f^n(L(z/\lambda^n))=f^n(f^k(y))=f^{n+k}(y),
\end{align*}
i.e., $w\in\bigcup_{n\geq 0} f^n(\Cv(f))$.

For the other inclusion, let $w\in f^n(\Cv(f))\setminus \Ev(f)$.
We want to show that there exists some $z\in L^{-1}(w)$ with $L'(z)=0$.
Again, we differentiate the iterated functional equation and obtain
\begin{align*}
 L'(z)=(f^{n+1})'(L(z/\lambda^{n+1}))\cdot
L'(z/\lambda^{n+1})\cdot\frac{1}{\lambda^{n+1}}
\end{align*}
for all $z\in\C$.
There exists some $y\in \Crit(f)$ such that $w=f^{n+1}(y)$.
Clearly, $y\notin \Ev(f)$ since $w\notin \Ev(f)$.
By Proposition \ref{omitvalues}, we have $y\notin \Ov(L)$,
so there exists $z\in\C$ with $y=L(z/\lambda^{n+1})$.
It follows by the chain rule that $L'(z)=0$,
and we have $w=f^{n+1}(y)=f^{n+1}(L(z/\lambda^{n+1}))=L(z)$,
which finishes the proof of (i).

We now prove (ii). For the composition $f\circ L$ one obtains
\begin{align*}
 \Sv(f\circ L)=S(f\vert_{f(\C)})\cup \cl{f(\Sv(L))} = \Sv(f)\cup f(\Sv(L)),
\end{align*}
since every Picard value of $f$ is also a singular value of $f$.
Let us abbreviate $S:=\Sv(f)\cup f(\Sv(L))$.
Since the composition
\begin{diagram}
\C\setminus L^{-1}(f^{-1}(S)) &\rTo^{L} &\C\setminus f^{-1}(S) &\rTo^{f} &\C\setminus S
\end{diagram}
is a covering map, it follows from (\ref{func_eq}) that
\begin{diagram}
\C\setminus \lambda^{-1}\cdot L^{-1}(f^{-1}(S)) &\rTo^{\lambda}
&\C\setminus L^{-1}(S) &\rTo^{L} &\C\setminus S
\end{diagram}
must be a covering map as well. Hence
\begin{align*}
\Sv(f)\cup f(\Sv(L))= S\supset \Sv(L\circ\lambda) = \Sv(L).
\end{align*}
The argument is commutative with respect to (\ref{func_eq}),
so we obtain the opposite inclusion, yielding the equality
$\Sv(L)=\Sv(f)\cup f(\Sv(L))$. But for a point $w\in\Sv(f)$, this implies
that $w\in\Sv(L)$, and so $f(w)\in f(\Sv(L))\subset \Sv(L)$.
By proceeding inductively, it follows for every $n\in\N$ that
$f^n(w)\in \Sv(L)$, hence $\Pv(f)\subset\Sv(L)$.

Let $w\in\C\setminus\Pv(f)$. Then there exists a disk $D\ni w$ such that
all inverse branches of all iterates of $f$ exist in $D$.
Let $v\in D$ and $z\in L^{-1}(v)$, and define $z_n:=z/\lambda^n$ and
$v_n:=L(z_n)$.
Let $g_n$ be the branch of $(f^n)^{-1}$ such that
$g_n(v)=v_n$ and let $D_n:=g_n(D)$.
By the Shrinking Lemma in \cite{lyubich_minsky}, it follows that the
diameter of the domains $D_n$ converges to $0$
(Actually, the statement in \cite{lyubich_minsky} is not phrased such that it
completely covers our setting but the proof gives what we require).
We choose a domain $U$ in which $L$ is injective.
Then for $n$ large enough, $D_n$ lies in $L(U)$.
Let $T$ be the branch of $L^{-1}$ that maps $D_n$ into $U$.
Then we have
\begin{align*}
 L\circ (\lambda^n\circ T\circ g_n)(z)
= f^n\circ L\circ \underbrace{(T\circ g_n)(z)}_{\in U}
=(f^n\circ g_n) (z)=z.
\end{align*}
Since $z$ is an arbitrarily chosen preimage of an arbitrary point in $D$,
all inverse branches of $L$ can be defined in $D$. Hence $w\in\C\setminus \Sv(L)$.
\end{proof}

If $f$ is a polynomial then $\Av(L)$
is contained in the union of attracting and parabolic periodic
cycles and the accumulation points of recurrent
critical points in $\J(f)$ \cite[Theorem 1]{drasin_okuyama}.
Depending on the location of the repelling fixed point $z_0$ relative
to $\F(f)$,
we can exclude certain attracting cycles of $f$ as asymptotic
values for $L$.
\begin{prop}
 Let $f$ be a polynomial and let $w\in\Av(L)$. If 
$w$ is an attracting periodic point of $f$,
then $z_0$ lies in the boundary of the immediate attracting basin of $w$.
\end{prop}

\begin{proof}
 Let $w$ be an attracting periodic point of $f$ of period $k$
and assume that $w$
is an asymptotic value of $L$. Then there exists a path $\gamma$
to $\infty$ for which
$\lim_{t\to\infty} L(\gamma(t))=w$. Since $w\in\F(f)$ and $\F(f)$ is
open, we can assume that $L(\gamma)\subset\F(f)$.
It follows from (\ref{eqn_iter_fct_eq})
that every path $\gamma_n(t):=\lambda^{-n}\cdot\gamma(t)$
is again an asymptotic path for $L$. Moreover,
the limit of $L$ along $\gamma_{nk}$ is contained in $f^{-nk}(w)$.
On the other hand, every such limit point must
lie in the set of attracting periodic points
\cite[Theorem 1]{drasin_okuyama},
hence it follows that
$\lim_{t\to\infty} L(\gamma_{nk}(t))=w$.
Furthermore, for every $\eps>0$ there exists $N_{\eps}\in\N$
such that for all $n\geq N_{\eps}$, the curve $L(\gamma_{nk})$
intersects $\D_{\eps}(z_0)$. Hence $z_0\in\partial A^{*}(w)$.
\end{proof}

Recall that a point $z\in\J(f)$ is called a
\emph{buried point} if it does not belong to the boundary
of any Fatou component (other that $\I(f)$).
\begin{cor}
If $f$ is a polynomial and $z_0$ is a buried point (of $f$) 
then $L$ has no asymptotic values.
\end{cor}

Linearizers can be very useful to construct entire or meromorphic functions
whose set of singular values satisfies certain conditions.
For instance, in \cite{mihaljevic-brandt}, there was given an example of an
entire function of finite order with no asymptotic values and
only finitely many critical values such that the ramification degree
on its Julia set was unbounded; the constructed function was
 a linearizer of a certain hyperbolic
quadratic polynomial. Here we want to show another interesting
example that can be constructed using linearizers, in this case
of a transcendental entire function $f$.

Let $f(z):=\mu\exp(z)$ where $\mu\in\C$ is chosen
such that $\bigcup_{n\geq 0} f^n(0)$ is dense in $\C$.
The existence of such parameters is well-known.
By \cite[Theorem 2]{langley_zheng}, the function $f$ has infinitely
many fixed points. Since $\Sv(f)=\lbrace 0\rbrace$,
at most one of them is non-repelling
\cite[Theorem 7]{bergweiler},
 so we can pick a repelling fixed point $z_0$ of
$f$. Let $L$ be a linearizer of $f$ at $z_0$. It follows
from the functional equation that $0$ is an omitted
value of $L$. By Proposition \ref{prop_sing_values_lin}, every
point $w_n:=f^n(0)$ is an asymptotic value of $L$. It
is also not hard to check that $L$ has a direct
singularity lying over each of the points $w_n$.
(For a clarification of terminology, see e.g. \cite{drasin_okuyama};
our last claim also follows from \cite[Theorem 1.4]{drasin_okuyama}, which is formulated
 for linearizers of rational maps only, but extends to
linearizers of transcendental entire maps with the same proof.)
Hence $L$ is a map for which the set of projections of direct singularities
(or direct asymptotic values) is dense in $\C$. This is optimal, since
by a theorem of Heins \cite{heins}, the set of
projections of direct singularities is always countable.

\section{Maximum and minimum modulus estimates}

In the remaining part of the article we
prove Theorem \ref{thm_main}.
From now on, we consider an arbitrary but fixed polynomial $p$ of degree $d\geq 2$,
hence $p$ can be written as
\begin{align*}
 p(z)=\sum_{i=0} ^d a_i z^i = a_0 + a_1 z\dots +a_d z^d , \ \;a_d\neq 0.
\end{align*}
For every $\eps>0$ we pick a constant $R_{\eps}\geq 1$
for which the conclusion of
Proposition \ref{prop_pol_1} is satisfied, and such that
$\eps_1<\eps_2$ implies $R_{\eps_1}>R_{\eps_2}$.
We assume that $p$ has a repelling fixed point $z_0$ with
multiplier $\lambda$, and we denote by $L$ a linearizer of $p$ at $z_0$.
We also pick a constant $R_L\geq 1$ such that
$\M(L,s)>s$ for all $s\geq R_L$.

\begin{lem}[Regularity of growth]
\label{lem_reg_growth_1}
 Let $\eps>0$, $r> \max\{R_{\eps},R_L\}$ and define
$k_{\eps}:=\log((1-\eps)\vert a_d\vert)$ and $K_{\eps}:=\log((1+\eps)\vert a_d\vert)$.
Then
\begin{align*}
 \prod_{i=0} ^{n-1}\left(\! d + \frac{\log k_{\eps}}{\log \M(L,\vert\lambda\vert^i r)} \right)
\!\leq\! \frac{\log\M(L,\vert\lambda\vert^n r)}{\log\M(L, r)}
\!\leq \!\prod_{i=0} ^{n-1}\!
\left(\! d + \frac{\log K_{\eps}}{\log \M(L,\vert\lambda\vert^i r)}\right)
\end{align*}
holds for all $n\in\N$.
\end{lem}

\begin{proof}
 Let $r$ be as assumed, and let $\tilde{z}\in \Circ_r$ be a point for which
$L(\tilde{z})\geq L(z)$ for all $z\in\Circ_r$. Let $\tilde{w}:=L(\tilde{z})$.
Then $\vert\tilde{w}\vert=\M(L,r)$ and it follows from
the functional equation (\ref{func_eq}) and Proposition \ref{prop_pol_1} that
\begin{eqnarray*}
 \log \M(L,\vert\lambda\vert r) &=&\log\M(p\circ L,r)
=\log\M(p,L(\Circ_r))
 \geq \log p(\tilde{w})\\
&\geq& \log\left( (1-\eps)\vert a_d\vert\cdot\vert\tilde{w}\vert^d\right)
= k_{\eps} + d\cdot\log\M(L,r),
\end{eqnarray*}
and
\begin{eqnarray*}
  \log \M(L,\vert\lambda\vert r) &=& \log\M(p,L(\Circ_r))
\leq \log \M(p,\M(L,r))\\
&\leq&\log\left( (1+\eps)\vert a_d\vert\cdot\M(L,r)^d\right)= K_{\eps} +d\cdot \log\M(L,r).
\end{eqnarray*}
Hence,
\begin{align*}
 \left( \frac{k_{\eps}}{\log \M(L,r) } + d\right)
\leq\frac{\log \M(L,\vert\lambda\vert r) }{\log \M(L,r)}
\leq\left(\frac{K_{\eps}}{\log \M(L,r) } + d\right).
\end{align*}
The statement now follows immediately from the fact that
\begin{align*}
\frac{\log \M(L,\vert\lambda\vert^n r) }{\log \M(L,r)}
=\frac{\log \M(L,\vert\lambda\vert^n r) }
{\log \M(L,\vert\lambda\vert^{n-1} r)}\cdot\;\dots\;\cdot
\frac{\log \M(L,\vert\lambda\vert r) }{\log \M(L,r)}.
\end{align*}
\end{proof}

\begin{lem}
\label{lem_reg_growth_2}
 For every $k\in\N$ there exists $R_k>0$ such that
for all $R>R_k$, $m\leq d^k$ and $n>k$,
\begin{align*}
 \M(L,r_n)>r_{n+1}^m,
\end{align*}
where the sequence $(r_n)$ is defined by
\begin{align*}
r_n:=\vert\lambda\vert^n\cdot\M^n(L,R).
\end{align*}
Moreover, we can choose
$R_1=2\cdot\max\bigg\{\log \vert a_d\vert,\log \frac{2}{\vert a_d\vert},
\log\vert\lambda\vert\bigg\}$.
\end{lem}

\begin{proof}
Let $\eps\in (0, 1/2)$ be arbitrary but fixed,
and let $R>\max\{R_L, R_{\eps}\}$.
 It follows from Lemma \ref{lem_reg_growth_1} with $r=\M^n(L,R)$ that
\begin{align*}
 \log\M(L,r_n)&=\log\M(L,\vert\lambda\vert^n\M^n(L,R))\\
&\geq \prod_{i=0} ^{n-1}\!\left(d + \frac{\log k_{\eps}}
{\log \M(L,\vert\lambda\vert^i \M^n(L,R))} \right)
\!\cdot\log\M(L,\M^n(L,R))\\
&\geq \left( d - \frac{\vert\log k_{\eps}\vert}{\log R}\right)^n \cdot \log\M^{n+1}(L,R).
\end{align*}
By definition,
\begin{align*}
 \log r_{n+1}^m &= m\log(\vert\lambda\vert^{n+1}\M^{n+1}(L,R))\\
&=m(n+1)\log\vert\lambda\vert + m\log \M^{n+1}(L,R).
\end{align*}
Define $c_R:=\frac{\vert\log k_{\eps}\vert}{\log R}$. We want to show that
there exists $R_k$ such that when $R>R_k$, $m\leq d^k$ and $n\geq k+1$, then
\begin{eqnarray*}
 \log\M^{n+1}(L,R)\cdot ((d-c_R)^n - m)>m(n+1)\log\vert\lambda\vert.
\end{eqnarray*}
Obviously, it is sufficient if the wanted constant $R_k$ satisfies
\begin{align*}
 \log R_k\cdot ((d-c_{R_k})^n - d^k)>d^k(n+1)\log\vert\lambda\vert
\end{align*}
for all $n\geq k+1$, and this is certainly true when we choose $R_k$
sufficiently large. We will omit the details since they follow from elementary calculus;
however, one can prove inductively that
every $R_k$ with $\log R_k
>\max\lbrace2\vert\log k_{\eps}\vert, \frac{2k}{d} \vert\log k_{\eps}\vert,
\frac{\sqrt{\e}\log\vert\lambda\vert}{(2-\sqrt{\e})(k+2)}\rbrace$ is sufficiently large.
Hence for $k=1$ we can choose
$R_1=2\max\{\vert\log \vert a_d\vert\vert, \vert\log \frac{1}{2}\vert a_d\vert\vert,
\log\vert\lambda\vert)$
since $\vert\log k_{\eps}\vert=\vert\log ((1-\eps)\vert a_d\vert)\vert$
and $\eps\in (0,1/2)$, and
since $\frac{\sqrt{\e}}{3\cdot (2-\sqrt{\e})}<2$.
\end{proof}

\begin{lem}[growth of minimum modulus]
\label{lem_growth_min_mod}
Suppose that $\J_{z_0}(p)=\{ z_0\}$ and let $m\in\N_{>1}$.
Then there exists
$R_m>0$ with the following property:
For every $r>R_m$ there is a simple closed curve $\Gamma^r$
separating $\Circ_r$ and $\Circ_{r^m}$
such that
\begin{align*}
 \m(L,\Gamma^r)>\M(L,r).
\end{align*}
\end{lem}

\begin{proof}
 Let $D$ be a disk around $0$ such that
$L\vert_D$ is conformal.
Let $\delta>0$ be sufficiently small such that $\D_{\delta}(z_0)\subset L(D)$.
By Proposition \ref{prop_curve_escapingset}
there exists a simple closed curve $\gamma_{\delta}\subset \D_{\delta}(z_0)\cap \I(p)$
 which surrounds $z_0$.
Since such a curve exists in the intersection of every arbitrarily
small neighbourhood of $z_0$ and $\I(p)$, we can assume w.l.o.g. that $D=\D$.
Let $\Gamma_{\delta}= L^{-1}(\gamma_{\delta})\cap\D$. Then $\Gamma_{\delta}$ is a simple
closed curve surrounding $0$.
Define
\begin{align*}
 s:=\min_{z\in\Gamma_{\delta}} \vert z\vert=\dist(0,\Gamma_{\delta})\quad \text{and}\quad
t:= \max_{z\in\Gamma_{\delta}} \vert z\vert.
\end{align*}
Obviously, both $s$ and $t$
are finite and positive constants.

Let $r>\left(\frac{\vert\lambda\vert\cdot t}{s}\right)^{\frac{1}{m-1}}$
be an arbitrary but fixed number. We define
$l(r)$ to be the unique integer for which
\begin{align*}
 \vert\lambda\vert^{l(r)-1} \leq r < \vert\lambda\vert^{l(r)}.
\end{align*}
Similarly, for the external radius $t$ of the curve $\Gamma_{\delta}$ we denote by
$l(t)$ the unique natural number for which
\begin{align*}
t\cdot \vert\lambda\vert^{l(t)} \leq r^m < t\cdot\vert\lambda\vert^{l(t) + 1}.
\end{align*}
(Note that the lower bound for $r$ implies that
$s\cdot\vert\lambda\vert^{l(t)}>r$.)
By taking logarithms we obtain the equivalent equations
\begin{align*}
l(r)-1\leq  \frac{\log r}{\log\vert\lambda\vert}<l(r)
\end{align*}
and
\begin{align*}
l(t)\leq  \frac{m\cdot\log r - \log t}{\log\vert\lambda\vert}<l(t) + 1.
\end{align*}
A combination of these two inequalities yields
\begin{align}
\label{eq_combi}
 m\cdot l(r) - \left( \frac{\log t}{\log\vert\lambda\vert} + m + 1\right)
<l(t)
< m\cdot l(r) - \frac{\log t}{\log\vert\lambda\vert}.
\end{align}

Let us fix an $\eps\in (0,1/2)$.
Let $j\in\N$ be minimal with the property that
$p^j(\gamma_{\delta})\subset \{z: \vert z\vert >R_{\eps}\}$.
Note that there is a unique integer $j$ with this property since $\gamma_{\delta}$
is a compact subset of $\I(p)$.
We define
\begin{align*}
 \Gamma^r:=\{ z\in\C: \lambda^{-l(t)}\cdot z\in \Gamma_{\delta}\}.
\end{align*}
Observe that $\Gamma^r$ separates $\Circ_r$ and $\Circ_{r^m}$.
In order to simplify the calculations,
let us consider the logarithms of the minimum and maximum modulus.
Using Proposition \ref{prop_pol_1}, these can be estimated in the following way:
\begin{eqnarray*}
\log \m(L,\Gamma^r) \!\!
&=&\!\! \log\m(p^{l(t)}\circ L, \Gamma_{\delta})
=\log \m(p^{l(t)}, \gamma_{\delta})
\geq \log\m(p^{l(t)-j}, R_{\eps})\\
&\geq& \!\!\log\{((1-\eps)\vert a_d\vert)^{q_{l(t)-j}(d)}\cdot R_{\eps}^{d^{l(t)-j}}\}\\
&=&\!\! q_{l(t)-j}(d)
\cdot \log ((1-\eps)\vert a_d\vert) + d^{l(t)-j}\cdot \log R_{\eps},\\[2ex]
\log \M(L,r)\!\!
&=&\!\!\log\M(p^{l(r)},L(\Circ_{r\cdot \vert\lambda\vert^{-l(r)}}))
\leq \log\M(p^{l(r)},R_{\eps})\\
&\leq& \!\!\log\{((1+\eps)\vert a_d\vert)^{q_{l(r)}(d)}\cdot R_{\eps}^{d^{l(r)}}\}\\
&\leq& \!\! q_{l(r)}(d)
\cdot \log ((1+\eps)\vert a_d\vert) + d^{l(r)}\cdot \log R_{\eps}.
\end{eqnarray*}
Equation (\ref{eq_combi}) yields the relation
$m\cdot l(r)-C  < l(t)< m\cdot l(r) + c$
with the constants
$C:=\log t/\log\vert\lambda\vert +m +1$
and $c:=\log t/\log\vert\lambda\vert$.
Furthermore, by Proposition \ref{prop_pol_1} we can estimate the polynomials
$q_{n+1}(d) = d^{n} + \ldots + d +1 = (d^{n+1}-1)/(d-1)$ by
$d^n\leq q_{n+1}(d)\leq d^{n+1}$. Together, we obtain
\begin{eqnarray*}
\log \m(L,\Gamma^r)
&>& d^{m\cdot l(r) - C-j-1}\cdot \log ((1-\eps)\vert a_d\vert)
+ d^{m\cdot l(r) -C - j}\cdot \log R_{\eps}\\
&=& d^{m\cdot l(r)}\cdot \frac{\log ((1-\eps)\vert a_d\vert R_{\eps}^d)}{d^{C+j+1}},\\[2ex]
 \log \M(L,r) &\leq& d^{l(r)} \cdot \log ((1+\eps)\vert a_d\vert)
+ d^{l(r)}\cdot \log R_{\eps}\\
&=& d^{l(r)}\cdot \log ((1+\eps)\vert a_d\vert R_{\eps})
\end{eqnarray*}
as new lower and upper bounds for the minimum and maximum modulus, respectively.
Since $m\geq 2$, it is sufficient to find a constant $R_m$ such that for all
$r>R_m$,
\begin{align*}
&\;& d^{2l(r)}\cdot \frac{\log ((1-\eps)\vert a_d\vert R_{\eps}^d)}{d^{C+j+1}}
&> d^{l(r)}\cdot \log ((1+\eps)\vert a_d\vert R_{\eps})\\
&\Longleftrightarrow& d^{l(r)}
&> \frac{\log ((1+\eps)\vert a_d\vert R_{\eps})}{\log ((1-\eps)\vert a_d\vert R_{\eps}^d)}
\cdot d^{C+j+1}=:l_{\eps}.
\end{align*}
Hence
$R_m:=\max\bigg\{\left(\frac{\vert\lambda\vert\cdot t}{s}\right)^{\frac{1}{m-1}},
\vert\lambda\vert^{\frac{\log l_{\eps}}{\log d}}\bigg\}$ is sufficiently large.
\end{proof}

\begin{proof}[Proof of Theorem \ref{thm_main}]
 Let us start with the case when $\J_{z_0}(p)\neq\{z_0\}$.
Assume that $A_R(L)$ is a spider's web for some sufficiently large $R$.
By definition,
there exists a sequence of bounded simply-connected domains $G_n$ such that
$G_n\subset G_{n+1}$, $\partial G_n\subset A_R(L)$ for $n\in\N$, and
$\bigcup G_n=\C$. We can assume w.l.o.g. that every $G_n$ contains $0$
(since this is true anyway for all sufficiently large $n$).

By Proposition \ref{prop_curve_escapingset}, for every $n\in\N$, the curve
$L(\partial G_n)$ intersects the filled Julia set of $p$.
Let $K>0$ be the radius of the smallest disk around $0$
which contains the (filled) Julia set of $p$. Then there
exists a sequence of points $w_n\in\partial G_n$ such that
$\vert L(w_n)\vert\leq K$. But this contradicts the assumption that
all points $z\in\partial G_n$ satisfy $\vert L(z)\vert\geq \M(L,R)$.

Let us now consider the situation when $\J_{z_0}(p)=\{z_0\}$.
By \cite[Theorem 8.1]{rs_2} it is sufficient to find
a sequence of bounded simply-connected domains $G_n$ such that
for all (sufficiently large) $n$,
\begin{align*}
 G_n\supset\{ z\in\C: \vert z\vert <\M^n(L,R)\}
\end{align*}
and
\begin{align*}
G_{n+1} \text{ is contained in a bounded component of }\C\setminus L(\partial G_n).
\end{align*}
Let $R_1$ be the constant from Lemma \ref{lem_reg_growth_2}, and
set $R:=\max \{R_L,R_1\}$.
For $n\in\N$ let $r_n:=\vert\lambda\vert^n\M^n(L,R)$
(see also Lemma \ref{lem_reg_growth_2}).
By Lemma \ref{lem_growth_min_mod},
there exists
a simple closed curve $\Gamma^{r_n}$ separating $\Circ_{r_n}$ and
$\Circ_{r_n^d}$ such that $\m(L, \Gamma^{r_n})>\M(L,r_n)$.
We define $G_n$ to be the interior of $\Gamma^{r_n}$.
Then every $G_n$ is a bounded simply-connected domain with
\begin{align*}
 G_n\supset \{ z\in\C:\vert z\vert <r_n\}\supset \{ z\in\C:\vert z\vert <\M^n(L,R)\}.
\end{align*}
Furthermore, it follows from Lemma \ref{lem_reg_growth_2} with $m=d$ that
\begin{align*}
 \m(L,\partial G_n)=\m(L,\Gamma^{r_n})>\M(L,r_n)>r_{n+1}^d
>\max_{z\in\partial G_{n+1}}\vert z\vert,
\end{align*}
hence $G_{n+1}$ is contained in a bounded component of $\C\setminus L(\partial G_n)$
and the claim follows.
\end{proof}

Note that Corollary \ref{cor_1} is an immediate consequence of
Theorem \ref{thm_main}.

\end{document}